\documentclass{article}

\usepackage{amsfonts}
\usepackage{amssymb}
\usepackage{amsmath}
\usepackage{amsthm}
\usepackage{ifthen}
\usepackage{enumerate}

\newcommand{\C}{{\mathbb C}}
\newcommand{\N}{{\mathbb N}}

\newcommand{\Q}{{\mathbb Q}}

\newcommand{\jac}{{\mathcal J}}
\newcommand{\hess}{{\mathcal H}}
\newcommand{\parder}[3][Default]{
	\frac{\partial \ifthenelse{\equal{#1}{Default}}{}{^{#1}}#2}{
              \partial #3 \ifthenelse{\equal{#1}{Default}}{}{^{#1}}}}
\newcommand{\tp}{^{\rm t}}
\newcommand{\grad}{\nabla}
\newcommand{\GL}{\operatorname{GL}}
\newcommand{\rk}{\operatorname{rk}}
\newcommand{\tr}{\operatorname{tr}}
\newcommand{\trdeg}{\operatorname{trdeg}}
\newcommand{\imp}{{\mathversion{bold}$\Rightarrow$ }}

\newcommand{\p}{\mathfrak{p}}

\theoremstyle{plain}
\newtheorem{theorem}{Theorem}[section]
\newtheorem*{hessestheorem*}{Hesse's (invalid) theorem}
\newtheorem{proposition}[theorem]{Proposition}
\newtheorem{lemma}[theorem]{Lemma}
\newtheorem{corollary}[theorem]{Corollary}
\newtheorem{problem}{Problem}
\newtheorem*{problemm}{Problem \ref{m}}
\theoremstyle{definition}

\theoremstyle{remark}
\newtheorem{remark}[theorem]{Remark}
\newtheorem{example}[theorem]{Example}
\theoremstyle{plain}

\allowdisplaybreaks
\numberwithin{equation}{section}

\title{Quasi-translations and singular Hessians}
\author{Michiel de Bondt}

\begin{document}

\maketitle

\begin{abstract}
In 1876 in \cite{gornoet}, the authors Paul Gordan and Max N{\"o}ther
classify all homogeneous polynomials $h$ in at most five variables
for which the Hessian determinant vanishes. For that purpose, they
study quasi-translations which are associated with singular 
Hessians. 

We will explain what quasi-translations are and formulate some elementary
properties of them. Additionally, we classify all quasi-translations with 
Jacobian rank one and all so-called irreducible homogeneous quasi-translations 
with Jacobian rank two. The latter is an important result of \cite{gornoet}.
Using these results, we classify all quasi-translations in dimension at most three
and all homogeneous quasi-translations in dimension at most four.

Furthermore, we describe the connection of quasi-translation with singular Hessians, 
and as an application, we will classify all polynomials in dimension
two and all homogeneous polynomials in dimensions three and four whose Hessian 
determinant vanishes. More precisely, we will show that up to linear terms, 
these polynomials can be expressed in $n-1$ linear forms, where $n$ is the 
dimension, according to an invalid theorem of Hesse.

In the last section, we formulate some known results and conjectures in connection 
with quasi-translations and singular Hessians.
\end{abstract}

\paragraph{Key words:} Quasi-translation, Hessian, determinant zero, homogeneous,
locally nilpotent derivation, algebraic dependence, linear dependence.

\paragraph{MSC 2010:} 14R05, 14R10, 14R20, 13N15.

\section{Quasi-translations}

Let $A$ be a commutative ring with $\Q$ and $n$ be a positive natural number. Write
$$
x = \left( \begin{array}{c} x_1 \\ x_2 \\ \vdots \\ x_n \end{array} \right)
  = (x_1, x_2, \ldots, x_n)
$$
for the identity map from $A^n$ to $A^n$, thus $x_1, x_2, \ldots, x_n$
are variables that correspond to the coordinates of $A^n$. 
A {\em translation} of $A^n$ is a map 
$$
x + c = \left( \begin{array}{c} x_1 + c_1 \\ x_2 + c_2 \\ 
               \vdots \\ x_n + c_n \end{array} \right)
$$
where $c \in A^n$ is a fixed vector. For a translation
$x + c$, we have that $x - c$ is the inverse polynomial map, because
$$
(x - c) \circ (x + c) = \left( \begin{array}{c} (x_1 + c_1) - c_1 \\ 
                               (x_2 + c_2) - c_2 \\ \vdots \\
                               (x_n + c_n) - c_n \end{array} \right) 
= \left( \begin{array}{c} x_1 \\ x_2 \\ \vdots \\ x_n \end{array} \right) = x
$$
is the identity map.

Inspired by this property, we define a {\em quasi-translation}
as a polynomial map
$$
x + H = \left( \begin{array}{c} x_1 + H_1 \\ x_2 + H_2 \\ 
               \vdots \\ x_n + H_n \end{array} \right)
      = \left( \begin{array}{c} x_1 + H_1(x) \\ x_2 + H_2(x) \\ 
               \vdots \\ x_n + H_n(x) \end{array} \right)
$$
such that $x - H$ is the inverse polynomial map of $x + H$, i.e.\@
$$
(x - H) \circ (x + H) = \left( \begin{array}{c} \big(x_1 + H_1(x)\big) - H_1(x + H) \\ 
                               \big(x_2 + H_2(x)\big) - H_2(x + H) \\ \vdots \\
                               \big(x_n + H_n(x)\big) - H_n(x + H) \\ \end{array} \right)
= x
$$
The difference between a translation $x + c$ and a quasi-translation $x + H$ is that
$c_i \in A$ for all $i$, while $H_i \in A[x] = A[x_1,x_2,\ldots,x_n]$ for all $i$.

For a regular translation $x + c$, we have that applying it $m$ times
comes down to the translation $x + mc$.
If $f \in \C[x]$ is an invariant of a regular translation $x + c$, i.e.\@
$f(x + c) = f(x)$, then $f(x + (m+1) c) = f(x + mc)$ follows by substituting
$x = x + mc$, whence by induction on $m$, $f(x + mc) = f(x)$ for all $m \in \N$.
Below, we show similar results for quasi-translations.

Since $(x - H) \circ (x + H) = x$, we see that
\begin{align*}
(x + mH) \circ (x + H) &= \big((m+1)x - m(x - H)\big) \circ (x + H) \\
&= (m+1)(x + H) - mx = x + (m+1)H
\end{align*}
Hence it follows by induction on $m$ that applying $x + H$ $m$ times comes 
down to the map $x + mH$ indeed. Consequently, if $f \in A[x]$ is an invariant of a 
$x + H$, i.e.\@ $f(x + H) = f(x)$, then
$$
f (x + mH) = f ((x + H) \circ \cdots \circ (x + H)) = f(x)
$$
follows for each $m \in \N$ by applying $f(x + H) = f(x)$ $m$ times.
Thus
\begin{equation} \label{fxmH}
f(x + H) = f(x) \Longrightarrow f(x + mH) = f(x) \mbox{ for all $m \in N$}
\end{equation}
With some techniques `on the shelf', we can improve \eqref{fxmH} to the following.

\begin{lemma} \label{vandermonde}
Assume that $x + H$ is a quasi-translation over $A$. Then
\begin{equation} \label{fxtH}
f(x + H) = f(x) \Longrightarrow f(x + tH) = f(x)
\end{equation}
where $t$ is a new indeterminate.
\end{lemma}

\begin{proof}
Let $d := \deg f$ and write $f(x+tH) - f = 
c_d t^d + c_{d-1} t^{d-1} + \cdots + c_1 t + c_0$. 
Since $f(x+mH) - f = 0$ for all $m \in \{0,1,2,\ldots,d\}$ 
on account of \eqref{fxmH},
$$
\left( \begin{array}{ccccc}
1 & 0 & 0^2 & \cdots & 0^d \\
1 & 1 & 1^2 & \cdots & 1^d \\
1 & 2 & 2^2 & \cdots & 2^d \\
\vdots & \vdots & \vdots & \ddots & \vdots \\
1 & d & d^2 & \cdots & d^d
\end{array} \right) \cdot \left( \begin{array}{c}
c_0 \\ c_1 \\ c_2 \\ \vdots \\ c_d
\end{array} \right) = \left( \begin{array}{c}
0 \\ 0 \\ 0 \\ \vdots \\ 0
\end{array} \right)
$$
The matrix on the left hand side is of Hadamard type and hence invertible,
because $A \supseteq \Q$. Consequently, $c_0 = c_1 = c_2 = \cdots = c_d = 0$. 
Hence $f(x+tH) = f$, as desired.
\end{proof}

But Gordan and N{\"o}ther did not use the term quasi-translation, and characterized
quasi-translations in another way. To describe it, we define the Jacobian matrix
of a polynomial map $H = (H_1, H_2, \ldots, H_m)$:
$$
\jac H := \left( \begin{array}{cccc}
\parder{}{x_1} H_1 & \parder{}{x_2} H_1 & \cdots & \parder{}{x_n} H_1 \\
\parder{}{x_1} H_2 & \parder{}{x_2} H_2 & \cdots & \parder{}{x_n} H_2 \\
\vdots & \vdots & \ddots & \vdots \\
\parder{}{x_1} H_m & \parder{}{x_2} H_m & \cdots & \parder{}{x_n} H_m 
\end{array} \right)
$$
Notice that $\jac H$ is a row vector if $H$ is a single polynomial, i.e.\@ $m = 1$.

If we see $H$ as a column vector, we can evaluate the matrix product $\jac H \cdot H$,
and the characterization of quasi-translations $x + H$ by Gordan and N{\"o}ther comes down to
\begin{equation} \label{JHH0}
\jac H \cdot H = (0^1, 0^2, \ldots, 0^n) 
\end{equation}
Here, the powers of zero are only taken to indicate the length of the zero vector on the 
right.

Gordan and N{\"o}ther derived the following under the assumption of \eqref{JHH0}. 
Take a polynomial $f \in A[x]$ such that $\jac f \cdot H = 0$. By the chain rule
\begin{align*}
\jac f(x + tH) \cdot H &= (\jac f)|_{x=x+tH} \cdot (I_n + t \jac H) \cdot H \\
&= (\jac f)|_{x=x+tH} \cdot H = \parder{}{t} f(x + tH)
\end{align*}
where $I_n$ is the unit matrix of size $n$. Here, $|_{x=\cdots}$ means substituting
$\cdots$ for $x$. Since $\jac f \cdot H = 0$, it follows from the above that 
\begin{equation} \label{qtinv}
\jac \big(f(x + tH) - f(x)\big) \cdot H = \parder{}{t} f(x + tH)
\end{equation}
Suppose that $t$ divides the right hand side of \eqref{qtinv} exactly $r < \infty$ times.
Then $t$ divides $f(x+tH) - f(x)$ more than $r$ times. Hence $t$ divides the left hand
side of \eqref{qtinv} more than $r$ times as well, which is a contradiction. 
So both sides of \eqref{qtinv} are zero. Since the right hand side of \eqref{qtinv} is zero 
and $A \supseteq \Q$, we get $f(x+tH) = f$, and that is what Gordan and N{\"o}ther derived from \eqref{JHH0}.

\begin{lemma} \label{gnlemma}
Assume $H$ is a polynomial map over $A$, such that $\jac H \cdot H = (0^1, 0^2, \ldots, 0^n)$. Then
\begin{equation} \label{JfH}
\jac f \cdot H = 0 \Longrightarrow f(x + tH) = f(x)
\end{equation}
for every $f \in A[x]$.
\end{lemma}

With lemmas \ref{vandermonde} and \ref{gnlemma}, we can prove the following.

\begin{proposition} \label{qtprop}
Let $A$ be a commutative ring with $\Q$ and $H: A^n \Rightarrow A^n$ be a polynomial map. 
Then the following properties are equivalent:
\begin{enumerate}[\upshape (1)]

\item $x + H$ is a quasi-translation, 

\item $H(x + tH) = H$, where $t$ is a new indeterminate,

\item $\jac H \cdot H = (0^1,0^2,\ldots,0^n)$.

\end{enumerate}
Furthermore,
\begin{equation} \label{fxtHeqv}
f(x + H) = f(x) \Longleftrightarrow f(x + tH) = f(x) \Longleftrightarrow \jac f \cdot H = 0
\end{equation}
holds for all $f \in A[x]$, and additionally
\begin{equation} \label{qtnilp}
(\jac H)|_{x=x-t\jac H} = (\jac H) + t (\jac H)^2 + t^2 (\jac H)^3 + \cdots
\end{equation}
if any of {\upshape (1)}, {\upshape (2)} and {\upshape (3)} is satisfied.
\end{proposition}

\begin{proof}
The middle hand side of \eqref{fxtHeqv} gives the left hand side by substituting
$t = 1$ and the right hand side by taking the coefficient of $t^1$.
Hence \eqref{fxtHeqv} follows from (1) and lemma \ref{vandermonde} and (3) and 
lemma \ref{gnlemma}.

By taking the Jacobian of (2), we get $(\jac H)|_{x=x+tH} \cdot (I_n + t \jac H) 
= \jac H$, which gives \eqref{qtnilp} after substituting $t = -t$.
Therefore, it remains to show that (1), (2) an (3) are equivalent.
\begin{description}

\item [(1) \imp (2)]
Assume (1). Since $x = (x - H) \circ (x + H) = x + H - H(x+H)$, we see that
$H(x + H) = H$, and (2) follows by taking $f = H_i$ in lemma \ref{vandermonde}.

\item [(2) \imp (1)]
Assume (2). Then
\begin{align} \label{ixtH}
(x - tH) \circ (x + tH) &= (x + tH) - tH(x + tH) \nonumber \\
&= x + tH - tH = x
\end{align}
which gives (1) after substituting $t = 1$.

\item[(2) \imp (3)]
Assume (2). By taking the coefficient of $t^1$ of (2), we get (3).

\item[(3) \imp (2)]
Assume (3). By taking $f = H_i$ in lemma \ref{gnlemma}, we get (2). \qedhere

\end{description}
\end{proof}

Notice that \eqref{ixtH} tells us that in some sense, $x + tH$ is 
a quasi-translation as well.

\begin{remark}
Let $D$ be the derivation $\sum_{i=1}^n H_i \parder{}{x_i}$. Then one can easily verify that
$$
\jac f \cdot H = 0 \Longleftrightarrow D f = 0
$$
and
$$
\jac H \cdot H = 0 \Longleftrightarrow D^2 x_1 = D^2 x_2 = \cdots = D^2 x_n = 0
$$
Hence quasi-translations correspond to a special kind of locally nilpotent derivations.
Furthermore, invariants of the quasi-translation $x + H$ are just kernel elements of
$D$. 

In addition, we can write $\exp(D)$ and $\exp (tD)$ for the automorphisms
corresponding to the maps $x + H$ and $x + tH$ respectively. But in order to make the 
article more readable for readers that are not familiar with derivations, we will omit the 
terminology of derivations further in this article.
\end{remark}

\section{Singular Hessians and Jacobians}

Now that we have had some introduction about quasi-translations, it is time to show 
how they are connected to singular Hessians. For that
purpose, we define the Hessian matrix of a polynomial $h \in \C[x]$ as follows:
$$
\hess h := \left( \begin{array}{cccc}
\parder{}{x_1}\parder{}{x_1} h & \parder{}{x_2}\parder{}{x_1} h & 
\cdots & \parder{}{x_n}\parder{}{x_1} h \\
\parder{}{x_1}\parder{}{x_2} h & \parder{}{x_2}\parder{}{x_2} h & 
\cdots & \parder{}{x_n}\parder{}{x_2} h \\
\vdots & \vdots & \ddots & \vdots \\
\parder{}{x_1}\parder{}{x_n} h & \parder{}{x_2}\parder{}{x_n} h & 
\cdots & \parder{}{x_n}\parder{}{x_n} h
\end{array} \right)
$$
A Hessian is a Jacobian which is symmetric 
with respect to the main diagonal. Hence each dependence between the columns of
a Hessian is also a dependence between the rows of it. 

For generality, we consider Jacobians which are not necessary Hessians.
Assume $G: \C^n \rightarrow \C^n$ is a polynomial map and $\det \jac G = 0$.
Let $\rk M$ denote the rank of a matrix $M$.
It is known that $\rk \jac H = \trdeg_{\C} \C(H)$, see e.g.\@ Proposition 1.2.9 of either 
\cite{arnobook} of \cite{homokema}. Hence there exists a nonzero polynomial 
$R \in \C[y] = \C[y_1,y_2,\ldots,y_n]$ such that 
\begin{equation} \label{gradrel}
R(G_1,G_2, \ldots, G_n) = 0
\end{equation}
Here, $y_1, y_2, \ldots, y_n$ are also variables that correspond to the 
coordinates of $\C^n$. Define
\begin{equation} \label{qthess}
H_i := \Big(\parder{R}{y_i}\Big)\Big|_{y=G} 
\end{equation}
for all $i$. By taking the Jacobian of \eqref{gradrel}, we get
$$
\jac 0 = \jac (R(G)) = (\jac_y R)|_{y = G} \cdot \jac G = H\tp \cdot \jac G
$$
where $H\tp = (H_1~H_2~\cdots~H_n)$ is the transpose of $H$.
Hence $H$ is a dependence between the {\em rows} of $\jac G$.
If $\jac G$ is a Hessian, $H$ is also a dependence between 
the {\em columns} of $\jac G$. Hence it seems a good idea to assume that 
$\jac G \cdot H = 0$. But for the sake of generality, we only assume that 
$\jac G \cdot \tilde{H} = 0$ for some polynomial vector $\tilde{H}$ 
that does not need to be equal to $H$.

If we write $\grad_y R$ for the transpose of $\jac_y R$, then $H = (\grad_y R)(G)$, 
and by the chain rule
\begin{align}
\jac H \cdot \tilde{H} = \jac (\grad_y R)(G) \cdot \tilde{H} = 
(\jac_y \grad_y R)|_{y=G} \cdot \jac G \cdot \tilde{H} = 0
\end{align}
whence $x + H$ is a quasi-translation on account of (3) $\Rightarrow$ (1)
of \ref{qtprop} if $H = \tilde{H}$. Thus we have proved the following. 

\begin{proposition} \label{qthesscor}
Assume $H_i$ is defined as in \eqref{qthess} for each $i$, with $R$ is as in 
\eqref{gradrel}, where $G: \C^n \rightarrow \C^n$ is a polynomial map.
Then $H$ is a dependence between the rows of $\jac G$.

If $\tilde{H}$ is
a dependence between the columns of $\jac G$, then $\jac H \cdot \tilde{H} = 0$.
In particular, if $H$ is a dependence between the columns of $\jac G$, then 
$x + H$ is a quasi-translation.
\end{proposition}

But $H$ defined as above may be the zero map. This is however not the case if
we choose the degree of $R$ as small as possible, without affecting \eqref{gradrel},
because $\parder{}{y_i} R$ is nonzero for some $i$ and of lower degree than $R$ 
itself.

\begin{example}
Let $p = x_1^2 x_3 + x_1 x_2 x_4 + x_2^2 x_5$, $h = p^r$ and $R = y_3 y_5 - y_4^2$. Then
$$
R(\grad h) = (r\, p^{r-1} x_1^2) (r\, p^{r-1} x_2^2) - (r\, p^{r-1} x_1 x_2)^2 = 0
$$
and $\grad_y R = (0,0,y_5,-2y_4,y_3)$. So 
$$
x + H := x + (\grad_y R)(\grad h) = x + r\, p^{r-1} (0,0,x_2^2, -2 x_1 x_2, x_1^2) 
$$
is a quasi-translation. Indeed, $x_1$, $x_2$ and $p$ are invariants of $x + H$.
\end{example}

\begin{example} \label{example}
Let $n \ge 6$ be even and 
$$
h = (A x_1 - B x_2)^2 + (A x_3 - B x_4)^2 + \cdots + (A x_{n-1} - B x_n)^2
$$
Then $(B y_{2i-1} + A y_{2i})|_{y = \grad h} = 0$ for all $i$, because
$$
B\parder{h}{x_{2i-1}} = 2AB(A x_{2i-1} - B x_{2i}) = -A\parder{h}{x_{2i}}
$$
for all $i$. Hence $(B^2 y_{2i-1}^2 - A^2 y_{2i}^2)|_{y = \grad h} = 0$ for all $i$
as well, and also
$$
R := \frac1{4AB} 
\big(B^2 (y_1^2 + y_3^2 + \cdots + y_{n-1}^2) - A^2 (y_2^2 + y_4^2 \cdots + y_n^2)^2 \big)
$$
satisfies $R(\grad h) = 0$. Now one can compute that
$$
x + H := x + \left(\begin{array}{c}
\parder{R}{y_1}(\grad h) \\
\parder{R}{y_2}(\grad h) \\
\parder{R}{y_3}(\grad h) \\
\parder{R}{y_4}(\grad h) \\
\vdots \\
\parder{R}{y_{n-1}}(\grad h) \\
\parder{R}{y_n}(\grad h)
\end{array} \right) = \left(\begin{array}{c}
B(A x_1 - B x_2) \\
A(A x_1 - B x_2) \\
B(A x_3 - B x_4) \\
A(A x_3 - B x_4) \\
\vdots \\
B(A x_{n-1} - B x_n) \\
A(A x_{n-1} - B x_n)
\end{array} \right)
$$
is a quasi-translation, and $\hess h \cdot H = (0^1,0^2,\ldots,0^n)$. 

Now let $a := x_1 x_4 - x_2 x_3$, $b := x_3 x_6 - x_4 x_5$, and
$G := (\grad h)|_{A = a, B = b}$.
Then $R(G) = 0$ as well. Thus $\tilde{H} := (\grad_y R)(G) = H|_{A = a, B = b}$ 
is a dependence between the rows of 
$$
\jac G = \big((\hess h) + (\jac_A \grad h) \cdot \jac a +
         (\jac_B \grad h) \cdot \jac b\big)
         \big|_{A = a, B = b}
$$
Since 
\begin{align*}
\lefteqn{\jac (x_{2i-1}x_{2j} - x_{2i}x_{2j-1}) \cdot \tilde{H}} \\
&= (x_{2i-1} \tilde{H}_{2j} - x_{2i} \tilde{H}_{2j-1}) + 
   (x_{2j} \tilde{H}_{2i-1} - x_{2j-1} \tilde{H}_{2i}) \\
&= (a x_{2i-1} - b x_{2i}) (a x_{2j-1} - b x_{2j}) + 
   (b x_{2j} - a x_{2j-1}) (a x_{2i-1} - b x_{2i}) \\
&= 0
\end{align*}
for all $i,j$, we have $\jac a \cdot \tilde{H} = \jac b \cdot \tilde{H} = 0$. Consequently,
\begin{align*}
\jac G \cdot \tilde{H} &= (\hess h)|_{A = a, B = b} \cdot \tilde{H} + 
            (\jac_a \grad h) \cdot \jac a \cdot \tilde{H} +
            (\jac_b \grad h) \cdot \jac b \cdot \tilde{H} \\
         &= \big((\hess h) \cdot H\big)\big|_{A = a, B = b} = (0^1,0^2,\ldots,0^n)
\end{align*}
Thus $x + \tilde{H}$ is a quasi-translation as well. In fact, it appears in 
\cite[Th.\@ 2.1]{debunk} as a homogeneous quasi-translation without linear invariants.
\end{example}

\section{Hesse's theorem}

A polynomial is {\em homogeneous} of degree $d$ if all its terms have degree $d$.
A homogeneous polynomial of degree one is called a {\em linear form}.
We call a quasi-translation $x + H$ {\em homogeneous} if there exists a $d \in \N$
such that each component $H_i$ of $H$ is either zero or homogeneous of degree $d$.

When Gordan and N{\"o}ther published their article in 1876, it was two years ago that
Otto Hesse had died. The role of Hesse is, that the starting point of Gordan and N{\"o}ther
was a wrong theorem of Hesse, which they proved in dimensions two, three and four, and 
disproved for all dimensions greater than four.

\begin{hessestheorem*}
Assume $n \ge 2$ and $h$ is a homogeneous polynomial in $x_1, x_2, \ldots, x_n$ 
over $\C$ whose Hessian matrix is singular. Then there is a constant vector
$c = (c_1, c_2, \ldots, c_n)$ such that
\begin{equation} \label{linrel}
c_1 \parder{h}{x_1} + c_2 \parder{h}{x_2} + \cdots + c_n \parder{h}{x_n} = 0
\end{equation}
or equivalently $\jac h \cdot c = 0$.
\end{hessestheorem*}

Gordan and N{\"o}ther first observe that \eqref{linrel} is equivalent to
the existence of a linear transformation which makes $h$ a polynomial in less
than $n$ variables, in other words, the existence of an invertible matrix $T$ 
such that
$$
h(Tx) \in \C[x_1,x_2,\ldots,x_{n-1}]
$$
In order to see that, we first compute the Jacobian matrix of $h(Tx)$
by way of the chain rule, where $|_{x=Tx}$ means substituting the vector $Tx$ for $x$.
\begin{equation} \label{jacgrad}
\jac \big(h(Tx)\big) = \big(\jac h(x)\big)\big|_{x=Tx} \cdot T
\end{equation}
If the $i$-th column of $T$ is equal to $c$, we have
$$
\parder{}{x_i} h(Tx) = \big(\jac h(x)\big)\big|_{x=Tx} \cdot c 
= \bigg(\sum_{j=1}^n c_j \parder{h}{x_j}\bigg)\bigg|_{x=Tx}
$$
whence for all $c_0 \in \C$,
\begin{equation} \label{gnrem}
\parder{}{x_i} h(Tx) = c_0 \Longleftrightarrow \sum_{j=1}^n c_j\parder{h}{x_j} = c_0
\end{equation}
By taking $i = n$ and $c_0 = 0$ in \eqref{gnrem}, we get the above remark of Gordan 
and N{\"o}ther.

We shall prove Hesse's theorem in dimensions two, three, and four. If
we drop the condition that $h$ is homogeneous, we can prove Hesse's
theorem in dimension two, but only for polynomials $h$ without linear terms.
For polynomials $h$ with linear terms, we can only get
$$
c_1 \parder{h}{x_1} + c_2 \parder{h}{x_2} + \cdots + c_n \parder{h}{x_n} = c_0 \in \C
$$
which holds for homogeneous $h$ in dimension one as well ($h = x_1$ is homogeneous,
but does not satisfy Hesse's theorem in dimension one). So we will prove the following.

\begin{theorem} \label{hesspos}
Hesse's theorem is true in dimension $n \le 4$, and for non-homogeneous
polynomials in dimension $n \le 2$.  
\end{theorem}

Before we prove our affirmative results about Hesse's theorem, we take a look at
the effects of affine transformations of $h$ and of removing linear terms of $h$.

View the vectors $x = (x_1,x_2,\ldots, x_n)$, $c = (c_1,c_2,\ldots, c_n)$ and
$\tilde{c} = (\tilde{c}_1, \tilde{c}_2, \ldots, \allowbreak \tilde{c}_n)$
as column matrices of variables and constants respectively. Let $T$ be
an invertible matrix of size $n$ over $\C$. Then
$$
h(Tx + c)
$$
is an affine transformation of $h$. Write $M\tp$ for
the {\em transpose} of a matrix $M$. Then $\tilde{c}\tp x$, which is the product of
a row matrix and a column matrix, is a matrix with only one entry, and
we associate it with the value of its only entry. If we define
$$
\tilde{h} := h(Tx + c) - \tilde{c}\tp x
$$
then it appears that $\det \hess \tilde{h} = 0$,
if and only if $\det \hess h = 0$. To see this, notice that
on account of \eqref{jacgrad},
$$
\jac \tilde{h} = \jac \big(h(Tx + c) - \tilde{c}\tp x\big)
= (\jac h)|_{x=Tx+c} \cdot T - \tilde{c}\tp
$$
The transpose of this equals
$$
\grad \tilde{h} = T\tp \cdot (\grad h)|_{x=Tx+c} - \tilde{c}
$$ 
where $\grad f$ stands for the transpose of $\jac f$. Taking the Jacobian of the above,
we get
$$
\hess \tilde{h} = \hess \big(h(Tx + c) - \tilde{c}\tp x\big)
= \jac \big(T\tp \cdot (\grad h)|_{x=Tx+c} - \tilde{c}\big)
= T\tp \cdot (\hess h)|_{x = Tx+c} \cdot T
$$
which is a singular matrix, if and only if $\hess h$ is. If $\det \hess h = 0$, then
$R(\grad h) = 0$ for some nonzero $R \in \C[y]$ on account of \eqref{gradrel}, 
and similarly $\tilde{R}(\grad \tilde{h}) = 0$. 

We shall describe a natural connection between $R$ and $\tilde{R}$.
If $R(\grad h) = 0$ for some $R \in \C[y]$, then
$$
R\Big((T\tp)^{-1} \big((T\tp \cdot (\grad h)|_{x=Tx+c} - \tilde{c}) + \tilde{c}\big)\Big) = 
\Big(R(\grad h)\Big)\Big|_{x=Tx+c} = 0
$$
thus $\tilde{R} := R((T\tp)^{-1}(y-\tilde{c}))$ satisfies $\tilde{R}(\grad \tilde{h}) = 0$.

If we define $H_i$ as in \eqref{qthess} with $G = \grad h$ for each $i$, then
we get $H = (\grad_y R) (\grad h)$, 
and $x + H$ is a quasi-translation on account of corollary \ref{qthesscor}.
In a similar manner, $x + \tilde{H}$ is a quasi-translation if
$\tilde{H} = (\grad_y \tilde{R}) (\grad \tilde{h})$, and again we describe a natural 
connection. By the chain rule,
$$
\jac_y \tilde{R} = \jac_y R\big((T\tp)^{-1}(y-\tilde{c})\big) = 
(\jac_y R)|_{y=(T\tp)^{-1}(y-\tilde{c})} \cdot (T\tp)^{-1}
$$
whence
$$
\grad_y \tilde{R} = T^{-1} (\grad_y R)|_{y=(T\tp)^{-1}(y-\tilde{c})}
$$
Combining this with $\grad \tilde{h} = T\tp \cdot (\grad h)|_{x=Tx+c} + \tilde{c}$, we
get that
$$ 
\tilde{H} = (\grad_y \tilde{R}) (\grad \tilde{h}) = T^{-1}H(Tx+c)
$$
thus $\tilde{H}$ is a linear conjugation of $H$ if $c = 0$. In general,
$$
x + \tilde{H} = x + T^{-1} H(Tx+c) = T^{-1} \Big(\big(Tx + c + H(Tx+c)\big) - c\Big)
$$
is an affinely linear conjugation of $x+H$. In the next section, we shall show that 
affinely linear conjugations of quasi-translations are again quasi-translations.

We are now ready to prove the following.

\begin{theorem} \label{regtrans}
Assume $h \in \C[x]$ and $R \ne 0$ satisfies \eqref{gradrel} 
and is of minimum degree. Define $H = (\grad_y R) (\grad h)$.

If $h$ is homogeneous, then $R$ and $H$ are homogeneous as well.

If the dimension of the linear span of the image of $H$ is at most one,
then $\deg R = 1$ and $x + H$ is a regular translation.
\end{theorem}

\begin{proof}
If $h$ is homogeneous, then \eqref{gradrel} is still satisfied if $R$ is replaced by
any homogeneous component of it. Since $R$ was chosen of minimum degree, we see
that $R$ and hence also $H$ is homogenous if $h$ is homogeneous. 

Assume that the dimension of the linear span of the image of $H$ is at most one.
Then there are $n-1$ independent vectors $c^{(i)} \in \C^n$ such that
$$
(c^{(1)})\tp H = (c^{(2)})\tp H = \cdots = (c^{(n-1)})\tp H = 0
$$
Suppose that $\deg R \ne 1$. Since $(c^{(i)})\tp H = ((c^{(i)}) \tp \grad R) (\grad h)$ 
for each $i$ and $R$ is of minimum degree,
we have $(c^{(i)}) \tp \grad R = 0$ for each $i$. Take an invertible matrix $T$ over 
$\C$ such that the $i$-th column of $(T\tp)^{-1}$ equals $c^{(i)}$ for each $i \le n-1$.
Then $(\jac R) \cdot c^{(i)} = 0$ for each $i \le n-1$. By applying \eqref{gnrem}
with $R(y)$ instead of $h(x)$ and $c_0 = 0$, we obtain 
$\parder{}{y_i} R\big((T\tp)^{-1} y\big) = 0$ for all $i \le n-1$, so
$$
\tilde{R}(y) := R\big((T\tp)^{-1} y\big) \in \C[y_n]
$$
Notice that for $\tilde{h} := h(Tx)$, we have $\tilde{R}(\grad \tilde{h}) = 
R\big((T\tp)^{-1} T\tp \grad h\big)\big|_{x=Tx} = 0$, whence
$$
\parder{}{x_n} \tilde{h}
$$
is algebraic over $\C$. 
This is only possible if this derivative is a constant $c_0$, 
which means that the left hand side of \eqref{gnrem} holds for $i = n$.
Thus $R$ can be chosen of degree one, and since that is the minimum possible degree
for $R$, we have $\deg R = 1$ and $H$ is constant.
\end{proof}

\begin{proof}[Proof of theorem \ref{hesspos}.]
In section \ref{hmgqt}, we will show that for quasi-translations $x + H$ in dimension two and
for homogeneous quasi-translations $x + H$ in dimension three, the linear span of the image of $H$
has dimension at most one. So the case $n \le 3$ of theorem \ref{hesspos} follows from theorem
\ref{regtrans}. Additionally, we will show in section \ref{hmgqt} that for homogeneous 
quasi-translations $x + H$ in dimension $n = 4$, the linear span of the image of $H$ has dimension at
most two. So the case $n = 4$ of theorem \ref{hesspos} follows from theorem
\ref{regtrans} and theorem \ref{regtranshmg} below.
\end{proof}

\begin{theorem} \label{regtranshmg}
Assume $h \in \C[x]$ and $R \ne 0$ satisfies \eqref{gradrel} 
and is of minimum degree. Define $H = (\grad_y R) (\grad h)$.

If $R$ is homogeneous, then the dimension of the linear span of the image of 
$H$ is not equal to two.
\end{theorem}

\begin{proof}
Assume that the linear span of the image of $H$ has dimension two. 
Then there are $n-2$ independent vectors $c^{(i)} \in \C^n$ such that
$$
(c^{(1)})\tp H = (c^{(2)})\tp H = \cdots = (c^{(n-2)})\tp H = 0
$$
Since $(c^{(i)})\tp H = ((c^{(i)}) \tp \grad R) (\grad h)$ for each $i$ 
and $R$ is of minimum degree,
we have $(c^{(i)}) \tp \grad R = 0$ for each $i$. Take an invertible matrix $T$ over 
$\C$ such that the $i$-th column of $(T\tp)^{-1}$ equals $c^{(i)}$ for each $i \le n-2$.
By applying \eqref{gnrem} with $R(y)$ instead of $h(x)$ and $c_0 = 0$, 
we see that $\parder{}{y_i} R\big((T\tp)^{-1} y\big) = 0$ for all $i \le n-2$, so
$$
\tilde{R}(y) := R\big((T\tp)^{-1} y\big) \in \C[y_{n-1},y_n]
$$
Notice that for $\tilde{h} := h(Tx)$, we have $\tilde{R}(\grad \tilde{h}) = 0$.

We will show below that Hesse's theorem holds for $\tilde{h}$ and hence also for $h$. 
Consequently, $\deg R = 1$ and $H$ is constant. This contradicts the assumption that 
the linear span of the image of $H$ has dimension two.

More precisely, we shall show that there are $c_{n-1}, c_n \in \C$, not both zero, such that
\begin{equation} \label{hlindep}
c_{n-1} \parder{}{x_{n-1}} \tilde{h} + c_{n} \parder{}{x_{n}} \tilde{h} = 0
\end{equation}
The case where $\parder{}{x_n} \tilde{h} = 0$ is trivial, so assume the opposite. Then
$$
\tilde{R} \bigg(\frac{\parder{}{x_{n-1}} \tilde{h}}{\parder{}{x_n} \tilde{h}},1\bigg)
= \tilde{R} \bigg(\frac{\grad \tilde{h}}{\parder{}{x_n} \tilde{h}}\bigg) 
= \frac{\tilde{R}(\grad \tilde{h})}{\big(\parder{}{x_n} \tilde{h}\big)^{\deg \tilde{R}}} 
 = 0
$$
Thus the quotient of $\parder{}{x_{n-1}} \tilde{h}$ and $\parder{}{x_n} \tilde{h}$
is algebraic over $\C$, which gives \eqref{hlindep}.
\end{proof}

If we take
$$
h = x_3 x_4 \qquad \mbox{and} \qquad R = y_1 y_3 + y_2 y_4
$$
we get $H = (x_4, x_3, 0, 0)$, and the span of image of $H$ has dimension two.
Thus the condition that $R$ has minimum degree is necessary in theorem 
\ref{regtrans}.

\section{Quasi-degrees}

Take $f \in \C[x]$ arbitrary and assume that $c$ is equal to the $i$-th 
standard basis unit vector. Then
$$
f (x + tc)|_{x_i = 0}\big|_{t=x_i} = f(x)
$$ 
for all $i$, whence the degree with respect to $x_i$ of $f$ is equal to
the degree with respect to $t$ of $f(x + tc)$. Based on this property, we define
the {\em quasi-degree} of a polynomial $f \in \C[x]$ with respect to a 
quasi-translation $x + H$ as the degree with respect to $t$ of $f(x + tH)$:
\begin{equation} \label{nudef}
\nu(f) := \deg_t f(x + tH)
\end{equation}
A property of $\nu$ that follows immediately from the definition is the following.
\begin{equation} \label{nuprod}
\nu(fg) = \nu(f) + \nu(g)
\end{equation}
The quasi-degree plays an important role in the study of quasi-translations, which the
proofs of the following propositions make clear.

\begin{proposition} \label{irred}
Assume $x + gH$ is a quasi-translation over $\C$, where $g \in \C[x]$ is nonzero. 
Then $x + H$ is a quasi-translation over $\C$ as well, and $\nu(g) = 0$. Furthermore, 
the invariants of $x + H$ are the same as those of $x + gH$.
\end{proposition}

\begin{proof}
From (1) $\Rightarrow$ (2) of proposition \ref{qtprop}, we deduce that
$H_i(x+tgH) \cdot g(x+tgH) = H_i \cdot g$. Substituting $t = g^{-1}t$ in it 
and using \eqref{nuprod} and $g \ne 0$, we obtain that
$$
\nu (H_i) \le \nu(g) + \nu (H_i) = \nu (g H_i) \le 0
$$
for each $i$, which is exactly $H(x+tH) = H$. Hence $x + H$ is 
a quasi-translation on account of (2) $\Rightarrow$ (1) of 
proposition \ref{qtprop}.

Thus it remains to be shown that the invariants of $x + H$ and $x + gH$ are the same.
Assume $f$ is an invariant of $x + H$. Then $f(x + tH) = f(x)$ on account of 
\eqref{fxtHeqv}, and by substituting $t = g$ we see that $f$ is an invariant of 
$x + gH$. The converse follows in a similar manner by substituting $t = g^{-1}$.
\end{proof}

Notice that proposition \ref{irred} above gives a tool to obtain quasi-translations
$x + H$ over $\C$ for which $\gcd\{H_1,H_2,\ldots,H_n\} = 1$ from arbitrary
quasi-translations $x + H$ over $\C$.

\begin{proposition} \label{qtconj}
Assume $x + H$ is a quasi-translation in dimension $n$ over $\C$, and 
$F$ is an invertible polynomial map in dimension $n$ over $\C$
with inverse $G$. Then
$$
G \circ (x + H) \circ F
$$
is a quasi-translation as well, if and only if $\nu(G_i) \le 1$ for all $i$.
In particular, if $T$ is an invertible matrix
of size $n$ over $\C$, we have that
$$
x + T^{-1} H(Tx) = T^{-1} \big(Tx + H(Tx)\big) = T^{-1}x \circ H \circ Tx
$$
is a quasi-translation as well.
\end{proposition}

\begin{proof}
Assume first that $\deg_t G(x+tH) \le 1$ for all $i$. Then we can write
$$
G(x+tH) = G^{(0)} + tG^{(1)}
$$
Notice that $G^{(0)} = G(x+tH)|_{t=0} = G$. Hence
$$
G \circ (x+tH) \circ F = G^{(0)}(F) + t G^{(1)}(F) = G(F) + t G^{(1)}(F) = x + t G^{(1)}(F)
$$
By substituting $t = 1$ on both sides, we obtain that $G \circ (x + H) \circ F
= x + G^{(1)}(F)$ and substituting $t = -1$ tells us that its inverse
$G \circ (x - H) \circ F$ is equal to $x - G^{(1)}(F)$. Thus $G \circ (x + H) \circ F$
is a quasi-translation.

Assume next that $G \circ (x + H) \circ F$ is a quasi-translation $x + \tilde{H}$. Then
$$
\tilde{H} = (G \circ (x + H) \circ F) - x = x - (G \circ (x - H) \circ F)
$$
Substituting $x = G(x + mH)$ in the above gives
$$
G\big(x + mH + H(x + mH)\big) - G(x + mH) = G(x + mH) - G\big(x + mH - H(x + mH)\big)
$$
Since $H(x + mH) = H$ follows by substituting $t = m$ in (2) of proposition 
\ref{qtprop}, we obtain
$$
G(x + (m+1)H) - G(x + mH) = G(x + mH) - G(x + (m-1)H)
$$
By induction on $m$, we get
$G(x + (m+1)H) - G(x + mH) = G(x + H) - G(x)$ for all $m \in \N$, whence
$$
G(x + \tilde{m}H) - G(x) = \sum_{m=0}^{\tilde{m}-1} G(x + (m+1)H) - G(x + mH) 
= \tilde{m} (G(x + H) - G(x))
$$
for all $\tilde{m} \in \N$.
Using similar techniques as in the proof of lemma \ref{vandermonde},
$$
G(x + t H) - G(x) = t (G(x + H) - G(x))
$$
follows. Hence $\deg_t G(x + t H) \le 1$, as desired.
\end{proof}

In example \ref{exampel} below, proposition \ref{qtconj} is used to
make a complicated quasi-translation from a simple one.

\begin{example} \label{exampel}
Let $n = 4$ and take $f = x_1 + x_2 x_4 - x_3^2$, $F = (f,x_2,x_3,x_4)$ and
$G = (2x_1 - f,x_2,x_3,x_4)$. Then $G$ is the inverse of $F$. 
Take $H = (0,x_1,x_1^2,x_1^3)$. 
Then $x + H$ is a quasi-translation which has a simple structure.
Furthermore, $\deg_t f(x+tH) \le \deg f = 2$, and the coefficient of
$t^2$ of $f(x + tH)$ is equal to
$$
H_2 H_4 - H_3^2 = x_1 x_1^3 - x_1^2 x_1^2 = 0
$$
Hence $\deg_t f(x+tH) \le 1$. It follows from proposition \ref{qtconj} that
$$
x + \tilde{H} := G \circ (x + H) \circ F 
$$
is a quasi-translation. Now
\begin{align*}
x_1 + \tilde{H}_1 &= G_1\big(F+H(F)\big) \\
&= G_1(f,x_2+f,x_3+f^2,x_4+f^3) \\
&= 2f - f(f,x_2+f,x_3+f^2,x_4+f^3) \\
&= f - (x_2 x_4 + x_2 f^3 + x_4 f + f^4) + (x_3^2 + 2f^2 x_3 + f^4) \\
&= x_1 - (f^3 x_2 - 2f^2 x_3 + f x_4)
\end{align*}
and
$$
\tilde{H} = \big({-(f^3 x_2 - 2f^2 x_3 + f x_4)},f,f^2,f^3\big)
$$
This quasi-translation has a more complicated structure than the one we started with.
In that fashion, the degrees of the components of $\tilde{H}$ are all different,
whence $x + \tilde{H}$ has no linear invariants.
\end{example}

\section{Quasi-translations with small Jacobian rank} \label{hmgqt}

Proposition \ref{qthmg} below gives a tool to obtain homogeneous quasi-translations 
over $\C$ from arbitrary quasi-translations $x + H$ over $\C$. Hence we can
obtain results about arbitrary quasi-translations by studying homogeneous ones.

\begin{proposition} \label{qthmg}
Assume $x + H$ is a quasi-translation over $\C$ in dimension $n$, and 
$$
d \ge \deg H := \max\{\deg H_1, \deg H_2, \ldots, \deg H_n\}
$$
If we define
$$
\tilde{x} := (x,x_{n+1}) \qquad \mbox{and} \qquad \tilde{H} := x_{n+1}^d \cdot 
\big(H(x_{n+1}^{-1}x), 0\big)
$$
then $\tilde{x} + \tilde{H}$ {\em homogeneous} quasi-translation (of degree $d$) over $\C$ 
in dimension $n + 1$. Furthermore, 
$$
\rk \jac H \le \rk \jac_{\tilde{x}} \tilde{H} \le \rk \jac H + 1
$$
\end{proposition}

\begin{proof}
Notice that $\tilde{H}$ is indeed polynomial and homogeneous of degree $d$.
\begin{enumerate}[(i)]

\item We show that $\tilde{x} + \tilde{H}$ is a quasi-translation 
in dimension $n+1$ over $\C$. On account of (3) $\Rightarrow$ (1) of 
proposition \ref{qtprop}, it suffices to show that 
$\jac_{\tilde{x}} \tilde{H} \cdot \tilde{H} = (0^1,0^2,\ldots,\allowbreak 0^{n+1})$. 
Since $\tilde{H}_{n+1} = 0$, this is equivalent to 
$$
\jac \tilde{H} \cdot x_{n+1}^d H(x_{n+1}^{-1}x) = (0^1,0^2,\ldots,0^{n+1})
$$
Using $\jac \tilde{H}_{n+1} = 0$ and factoring out $x_{n+1}^{2d-1}$, we see that
it suffices to show that 
$$
\jac \big(H(x_{n+1}^{-1}x)\big) \cdot x_{n+1} H(x_{n+1}^{-1}x) = (0^1,0^2,\ldots,0^n)
$$
This is indeed the case, because the chain rule tells us that
\begin{align*}
(0^1,0^2,\ldots,0^n) &= (\jac H \cdot H)_{x = x_{n+1}^{-1}x} \\
  &= (\jac H \cdot x_{n+1}^{-1} \cdot x_{n+1} H)_{x = x_{n+1}^{-1}x} \\
  &= \jac \big(H(x_{n+1}^{-1}x)\big) \cdot x_{n+1} H(x_{n+1}^{-1}x) \qedhere
\end{align*}

\item We show that $\rk \jac H \le \rk \jac_{\tilde{x}} \tilde{H} \le \rk \jac H + 1$. 
It is known that $\rk \jac H = \trdeg_{\C} \C(H)$, see e.g.\@ Proposition 1.2.9 of either 
\cite{arnobook} of \cite{homokema}. Hence it suffices to show that 
$\trdeg_{\C} \C(H) \le \trdeg_{\C} \C(\tilde{H}) \le \trdeg_{\C} \C(H) + 1$.

If $R(\tilde{H}) = 0$ for some polynomial $R \in \C[y]$, then substituting $x_{n+1} = 1$ gives
$R(H) = 0$. Hence $\trdeg_{\C} \C(H) \le \trdeg_{\C} \C(\tilde{H})$.
If $R(H) = 0$ for some polynomial $R \in \C[y]$, say of degree $r$, then
$\tilde{R} := y_{n+1}^r R(y_{n+1}^{-1} y)$ satisfies $\tilde{R}(\tilde{H},x_{n+1}^d) = 0$.
Hence $\trdeg_{\C} \C(\tilde{H}) \le \trdeg_{\C} \C(H) + 1$. \qedhere

\end{enumerate}
\end{proof}

Gordan and N{\"o}ther used techniques of algebraic geometry to obtain results
about homogeneous quasi-translation. They found the following property of 
homogeneous quasi-translations $x + H$:
\begin{equation} \label{Htp}
H(tH) = 0
\end{equation}
This equality can be obtained by looking at the leading coefficient of $t$
in $H(x + tH) = H$, which is (2) of proposition \ref{qtprop}. 

Below we will give algebraic proofs of the results of Gordan and N{\"o}ther about 
homogeneous quasi-translations with Jacobian rank at most $2$, especially theorem 
\ref{qthrk2} below.

\begin{theorem} \label{qthrk1}
Assume $x + H$ is a homogeneous quasi-translation in dimension $n$ over $\C$.
Then $\rk \jac H \le 1$, if and only if $H = gc$ for some $g \in \C[x]$ and
a vector $c \in \C^n$. If $H$ is not of the above form, then $2 \le \rk \jac H \le n-2$. 
In particular, $n \ge 4$ in that case.
\end{theorem}

\begin{proof}
If $H = gc$, then $\jac H = c \cdot \jac g$ and therefore $\rk \jac H \le 1$.
On the other hand, if $\rk \jac H = 0$, then $H = 0 = gc$ for some $g \in \C[x]$ and
a vector $c \in \C^n$, thus assume that $\rk \jac H = 1$. Then there exists
a $j \le n$ such that $g := H_j \ne 0$. If for all $i \le n$ and for each $f \in \C[x]$, 
$f$ divides $H_i$ at least as many times as it divides $g$, then $H = gc$ for some 
$g \in \C[x]$ and a vector $c \in \C^n$. So take any $i \le n$ and any $f \in \C[x]$.

It is known that $\rk \jac H = \trdeg_{\C} \C(H)$, see e.g.\@ Proposition 1.2.9 of either 
\cite{arnobook} of \cite{homokema}. Since $\trdeg_{\C} \C(H) = \rk \jac H = 1$, there
exists a nonzero polynomial $R \in \C[y_1,y_2]$ such that $R(H_i,H_j) = 0$.
Since $H$ is homogeneous, we can replace $R$ by one of its homogeneous components, 
so we may assume that $R$ is homogeneous. If $f$ divides $H_i$ fewer times than it 
divides $g = H_j$, then the number of times that $f$ divides $R(H_i,H_j)$ is determined 
by the nonzero term of lowest degree with respect to $y_2$ of $R$, which leads to a
contradiction. So $H = gc$ for some $g \in \C[x]$ and a vector $c \in \C^n$, if and only
if $\rk \jac H \le 1$.

Assume next that $H$ is not of the form $gc$ for any $g \in \C[x]$ and any vector 
$c \in \C^n$. Let $g = \gcd\{H_1,H_2,\ldots,H_n\}$ and define $\tilde{H} = g^{-1}H$.
On account of proposition \ref{irred}, $x + \tilde{H}$ is a homogenous quasi-translation 
as well, and by assumption, there is no $c \in \C^n$ such that $\tilde{H} = c$, i.e.\@
$\tilde{H}$ is not constant. From \eqref{Htp}, we deduce that $\tilde{H}_i(\tilde{H}) = 0$ 
for each $i$. Hence also $\tilde{H}_i(H) = 0$ for each $i$. 

Now suppose that $\rk \jac H = n-1$. Then $\trdeg_{\C} \C(H) = n-1$ as well, so the prime ideal 
$\p := \{ R \in \C[y] \mid R(H) = 0\}$ has height $n - (n-1) = 1$. Since $\C[y]$ is a unique 
factorization domain, it follows that $\p$ is principal.
This contradicts $\gcd\{\tilde{H}_1,\tilde{H}_2,\ldots, \tilde{H}_n\} = 1$, because $\tilde{H}_i \in \p$
for each $i$. So $2 \le \rk \jac H \le n-2$ if $H$ is not of the form $gc$ for some $g \in \C[x]$ and a 
vector $c \in \C^n$.
\end{proof}

Theorem \ref{qthrk2} is somewhat deeper and based on techniques in the paper \cite{gornoet} by 
Gordan and N{\"o}ther, see also \cite{gnlossen}. \cite[Th.\@ 3.6 (iii)]{hmgqt5dim} contains another
proof of the assertion that $s \ge 2$ in theorem \ref{qthrk2} below, which is sufficient to 
obtain theorem \ref{hesspos}.

\begin{theorem}[Gordan and N{\"o}ther] \label{qthrk2}
Assume $x + H$ is a homogeneous quasi-translation of degree $d$ over $\C$ such that 
$\rk \jac H = 2$. Then the linear span of the image of $H$ has dimension $n-2$ at most.

More precisely, if $H_1 = H_2 = \cdots = H_s = 0$ and $H_{s+1}, H_{s+2}, \ldots, H_n$
are linearly independent over $\C$, then $s \ge 2$ and $g^{-1} H_i \in \C[x_1,x_2,\ldots,x_s]$ 
for each $i$, where $g = \gcd\{H_1, H_2, \ldots, H_n\}$.
\end{theorem}

\begin{proof}
Let $g := \gcd\{H_1, H_2, \ldots, H_n\}$, and write $V(H)$ and $V(g^{-1} H)$ for the common 
zeros of $H_1,H_2,\ldots,H_n$ and $g^{-1}H_1,g^{-1}H_2,\ldots,g^{-1}H_n$ respectively. 
By replacing $H$ by $T^{-1}H(Tx)$ for a suitable $T \in \GL_n(\C)$, we can obtain that
$H_1 = H_2 = \cdots = H_s = 0$ and that $H_{s+1}, H_{s+2}, \ldots, H_n$ are linearly independent 
over $\C$. On account of proposition \ref{qtconj}, $x + H$ stays a quasi-translation. Furthermore, 
the equality $\rk \jac H = 2$ and the dimension of the linear span of the image of $H$ are preserved.

So we may assume that $H_1 = H_2 = \cdots = H_s = 0$ and $H_{s+1}, H_{s+2}, \ldots, H_n$ are 
linearly independent over $\C$ for some $s \ge 0$. On account of theorem \ref{qthrk1}, $g^{-1} H$ 
is not constant because $\rk \jac H = 2$. It is known that $\rk \jac H = \trdeg_{\C} \C(H)$, see 
e.g.\@ Proposition 1.2.9 of either \cite{arnobook} of \cite{homokema}. 
Since $R(H) = 0 \Leftrightarrow R(g^{-1}H) = 0$ for homogeneous 
and hence any $R \in \C[y]$, it follows that 
$$
\rk \jac (g^{-1}H) = \trdeg_{\C} \C(g^{-1}H) = \trdeg_{\C} \C(H) = \rk \jac H = 2
$$
Furthermore,
$g^{-1}H_{s+1}, g^{-1}H_{s+2}, \ldots, g^{-1}H_n$ are linearly independent over $\C$ as well.
On account of proposition \ref{irred}, we may assume that $g = 1$ and therefore $\dim V(H) \le n-2$. 

By theorem \ref{qthrk1}, we deduce from $\rk \jac H = 2$ that there is an $i$ such that
$x_1^d \nmid H_i$. So if $H_i \in \C[x_1,x_2,\ldots,x_s]$ for all $i$, then $s \ge 2$.
We prove $H_i \in \C[x_1,x_2,\ldots,x_s]$ for all $i$ by showing that generic linear combinations of 
$H_1, H_2, \allowbreak \ldots, H_n$ are contained in $\C[x_1,x_2,\ldots,x_s]$. 
The genericity condition on the linear combinations $\alpha\tp H$ with $\alpha \in \C^n$,
is that the intersection of the hyperplane of zeroes of $\alpha\tp y$ with a fixed finite set of
so-called exclusive lines through the origin must be trivial. 

From lemma \ref{phiL} below, it follows that each line $L$ through the origin and another 
point in the image of $H$ lies entirely in the image of $H$. Furthermore, for each such line 
$L$, there exists a polynomial $\phi^{(L)}$, which we can take square-free,
such that $H^{-1}(L \setminus \{0\}^n)
\subseteq V(\phi^{(L)}) \subseteq H^{-1}(L)$ on account of (i) of lemma \ref{phiL}, where 
$V(\phi^{(L)})$ is the set of zeroes of $\phi^{(L)}$. We call a line
$L$ in the image of $H$ {\em exclusive}, if there are
only finitely many other such lines $L'$ for which $\deg \phi^{(L)} \le \deg \phi^{(L')}$.
By considering an exclusive line $L$ for which $\deg \phi^{(L)}$ is minimum, we see that 
there can only be finitely many exclusive lines indeed. 

So let us assume the genericity condition that the set $S$ of zeroes of $\alpha\tp y$ intersects 
all exclusive lines in the origin only. If $S$ contains infinitely many lines $L$ 
in the image of $H$, then it follows from (ii) of lemma \ref{phiL} below that $S$ contains 
the image of $H$ as a whole, so that $\alpha\tp H = 0 \in \C[x_1,x_2,\ldots,x_s]$. 
So assume that $S$ contains only finitely many lines $L$ in the image of $H$, say 
$L_1, L_2, \ldots, L_m$, where $m \ge 0$.

Notice that a zero of $\alpha\tp H$ is a zero of either $V(H)$ or $\phi^{(L_k)}$ for some $k \le m$.
From the Nullstellensatz, it follows that the squarefree part $\sigma$ of $\alpha\tp H$
is a divisor of $H_i \phi^{(L_1)} \phi^{(L_2)} \cdots \phi^{(L_m)}$ for each $i$. Since $g = 1$,
we deduce that $\sigma$ is already a divisor of $\phi^{(L_1)} \phi^{(L_2)} \cdots \phi^{(L_m)}$. 
It suffices to prove that $\sigma \in \C[x_1,x_2,\ldots,x_s]$ and we will do that by showing 
that $\phi^{(L_k)} \in \C[x_1,x_2,\ldots,x_s]$ for each $k$.

Since $L_k$ is not exclusive, there are infinitely many lines $L' \ne L_k$ 
in the image of $H$, for which $\deg \phi^{(L_k)} \le \deg \phi^{(L')}$. 
Hence it follows from (ii) of lemma \ref{phiL} below that the linear span 
$S'$ of these lines $L'$ contains the image of $H$ as a whole. So $S' = \{0\}^s \times \C^{n-s}$ 
by definition of $s$.

In order to prove that $\phi^{(L_k)} \in \C[x_1,x_2,\ldots,x_s]$, it suffices to show that
$\jac \phi^{(L_k)} \cdot e_i = 0$ for all $i > s$. So take any $i > s$. Since $S'$ contains the image 
of $H$, we can write $e_i$ as a linear combination of several $c' \in S'$ such that $c' \in L'$
for some line $L' \ne L_k$ in the image of $H$, 
for which $\deg \phi^{(L_k)} \le \deg \phi^{(L')}$. Hence it suffices to show that
$\jac \phi^{(L_k)} \cdot c' = 0$ for every such $c'$. We will prove below that
\begin{equation} \label{eq1}
\phi^{(L')} \mid H_i \cdot \jac \phi^{(L_k)} \cdot c' 
\end{equation}
for each $i$ and every pair $c' \in L'$ as above. Since $g = 1$, 
we see that \eqref{eq1} implies $\phi^{(L')} \mid \jac \phi^{(L_k)} \cdot c'$, which gives
$\jac \phi^{(L_k)} \cdot c' = 0$ because $\deg \phi^{(L_k)} \le \deg \phi^{(L')}$.

So it remains to prove \eqref{eq1}.  Since $\phi^{(L_k)}$ is square-free and 
$V(\phi^{(L_k)}) \subseteq H^{-1}(L_k) \subseteq H^{-1}(S)$, it follows from the 
Nullstellensatz that $\phi^{(L_k)} \mid \alpha\tp H$. Hence $\deg_t \phi^{(L_k)} (x + tH)
\le \deg_t (\alpha\tp H (x + tH)) = 0$. So $\phi^{(L_k)} (x + t H) = \phi^{(L_k)}(x)$. On
account of \eqref{fxtHeqv} in proposition \ref{qtprop},
\begin{equation} \label{eq2}
\jac \phi^{(L_k)} \cdot H = 0
\end{equation}
Let $\theta \in V(\phi^{(L')})$ and $c := H(\theta)$. If $c$ is the zero vector, then 
$H_i(\theta) = 0$. If $c$ is not the zero vector, then $c \in L' \setminus \{0\}$, 
so $c'$ is a scalar multiple of $c$, and $(\jac \phi^{(L_k)})|_{x = \theta} \cdot c' = 0$ 
on account of \eqref{eq2} in this case. Consequently, $\theta$ is a zero of 
$H_i \cdot \jac \phi^{(L_k)} \cdot c'$ in any case. Now \eqref{eq1} follows from the 
Nullstellensatz, because $\theta$ was an arbitrary zero of $\phi^{(L')}$, which is square-free.
\end{proof}

\begin{lemma} \label{phiL}
Let $H$ be a homogeneous polynomial map of degree $d$ over $\C$, such that $\rk \jac H = 2$. 
Then the image of $H$ consists of a union of lines through the origin.
\begin{enumerate}[\upshape (i)] 
 
\item For each line $L$ through the origin in the image of $H$, there exists a polynomial 
$\phi^{(L)}$ such that $H^{-1}(L \setminus \{0\}) \subseteq V(\phi^{(L)}) \subseteq H^{-1}(L)$.

\item If a linear subspace $S$ of $\C^n$ contains infinitely many lines through the origin 
in the image of $H$, then $S$ contains the image of $H$ as a whole.

\end{enumerate}
\end{lemma}

\begin{proof}
Since $H$ is homogeneous, say of degree $d$, we have $\lambda H(\theta) = 
H(\sqrt[d]{\lambda} \theta)$. Consequently, the image of $H$ consists of a union of 
lines through the origin.
\begin{enumerate}[\upshape (i)] 
 
\item Take any nonzero $c \in L$ and let $W := \{\alpha \in \C^n \mid \alpha\tp c = 0\}$.
Define $\phi^{(L)} := \gcd\{ \alpha\tp H \mid \alpha \in W\}$. If $\phi^{(L)}(\theta) = 0$ 
for some $\theta \in \C^n$, then $\alpha\tp H(\theta) = 0$ for all $\alpha \in W$, which
implies that $H(\theta) \in L$. Hence $V(\phi^{(L)}) \subseteq H^{-1}(L)$ indeed.

So it remains to prove that $H^{-1}(L \setminus \{0\}^n) \subseteq V(\phi^{(L)})$. This is 
trivial if $\phi^{(L)} = 0$, so assume that $\phi^{(L)} \ne 0$. By definition of $\phi^{(L)}$, 
there exists an $\alpha \in W$ such that $\alpha\tp H \ne 0$. Furthermore, 
$\gcd \{(\phi^{(L)})^{-1} \beta\tp H \mid \beta \in W\} = 1$. Hence for every irreducible divisor $f$ of
$\alpha\tp H$, the set $\{ \beta \in W \mid (\phi^{(L)})^{-1} \beta\tp H$ is 
divisible by $f\}$ is a proper linear subspace of $W$. Since $\alpha\tp H$ has only finitely many 
irreducible divisors, we can choose $\beta \in W$ such that $f \nmid (\phi^{(L)})^{-1} \beta\tp H$ 
for every irreducible divisor $f$ of $\alpha\tp H$. In other words, 
$\gcd\{\alpha\tp H, (\phi^{(L)})^{-1} \beta\tp H\} = 1$.

We shall prove that $H^{-1}(L \setminus \{0\}^n) \subseteq V(\phi^{(L)})$ by showing that
for each $i \le n$, $\theta \notin V(\phi^{(L)})$ and $\theta \in H^{-1}(L)$ imply
$\theta \in H^{-1}(\{0\}^n)$. In other words, we show for each $i \le n$ that
$\phi^{(L)}(\theta) \ne 0$ and $H(\theta) \in L$ imply $H_i(\theta) = 0$. 
So let us take any $i \le n$ and suppose that $\phi^{(L)}(\theta) \ne 0$ and $H(\theta) \in L$. 
Then $\alpha\tp H(\theta) = \beta\tp H(\theta) = 0$, so there exists an irreducible divisor $f$ of 
$\beta\tp H$ such that $f(\theta) = 0$, but $f \nmid \phi^{(L)}$ because $\phi^{(L)}(\theta) \ne 0$.
So $f \mid (\phi^{(L)})^{-1} \beta\tp H$. From $\gcd\{\alpha\tp H, (\phi^{(L)})^{-1} \beta\tp H\} = 1$,
it follows that $f \nmid \alpha\tp H$. If there exists a homogeneous polynomial in $H_i$ and 
$\alpha\tp H$ which is divisible by $f$ and monic with respect to $H_i$, then we can deduce that
$H_i$ is contained in the radical of $(\alpha\tp H,f)$, which gives $H_i(\theta) = 0$ because
$\alpha\tp H(\theta) = f(\theta) = 0$.

So it remains to show that a homogeneous polynomial as above exists.
It is known that $\rk \jac H = \trdeg_{\C} \C(H)$, see e.g.\@ Proposition 1.2.9 of either 
\cite{arnobook} of \cite{homokema}.
Since $\trdeg_{\C} \C(H) = \rk \jac H = 2$, there exists a nonzero polynomial 
$R \in \C[y_1,y_2,y_3]$ such that $R(H_i, \alpha\tp H, \beta\tp H) = 0$. Since $H$ is homogeneous, we
can replace $R$ by any of its homogeneous components, so we may assume that $R$ is homogeneous.
Furthermore, we can replace $R$ by (at least) one of its irreducible factors, so we may assume that $R$
is irreducible as well. 

View $R$ as a polynomial in $y_3$ over $\C[y_1,y_2]$ and take for $R_0$ 
the coefficient of $y_3^0$ of $R$. If $R$ is linear, then $R_0 \ne 0$ because 
$R_0(H_i, \alpha\tp H)$ cancels out a scalar multiple of $\beta\tp H \ne 0$ in 
$R(H_i, \alpha\tp H, \beta\tp H) = 0$ if $R_0 \ne R$. 
If $R$ is not linear, then $R_0 \ne 0$ because $R$ is irreducible. So $R_0 \ne 0$
in any case. Since 
$$
f \mid \beta\tp H \mid R(H_i, \alpha\tp H, \beta\tp H) - R_0 (H_i, \alpha\tp H) = -R_0 (H_i, \alpha\tp H)
$$
and $f \nmid \alpha\tp H$, we get the homogeneous polynomial in $H_i$ and $\alpha\tp H$ 
that we need, i.e.\@ divisible by $f$ and monic with respect to $H_i$, if we remove 
factors $\alpha\tp H$ from $R_0 (H_i, \alpha\tp H)$.

\item Suppose that $S$ contains infinitely many lines through the origin in the image of $H$.
Since $S$ is a zero set of linear forms in $y$, $H^{-1}(S)$ is a zero set of
linear forms in $H$. Suppose $\alpha\tp H$ is any of these linear forms. It suffices to
prove that $H^{-1}(S) = \C^n$ and we do that by showing that $\alpha\tp H = 0$.

Choose $\phi^{(L)}$ as in (i) square-free for each line $L \subseteq S$ through the origin 
in the image of $H$. From the Nullstellensatz, it follows that $\phi^{(L)} \mid \alpha\tp H$ 
for each line $L \subseteq S$ through the origin in the image of $H$. Hence the least common multiple $f$
of the polynomials $\phi^{(L)}$, with $L$ as such, exists as a divisor of $\alpha\tp H$. 
Furthermore, $f$ can be written as the least common multiple of finitely many $\phi^{(L)}$ with 
$L$ as above. This contradicts the assumption that there are infinitely many 
lines $L \subseteq S$ through the origin in the image of $H$, so $\alpha\tp H = 0$. \qedhere

\end{enumerate}
\end{proof}

\begin{theorem} \label{qtrk1}
Assume $x + H$ is a quasi-translation over $\C$ such that 
$\rk \jac H = 1$. Then the linear span of the image of $H$ has dimension $n-1$ at most.

More precisely, if $H_1 = H_2 = \cdots = H_s = 0$ and $H_{s+1}, H_{s+2}, \ldots, H_n$
are linearly independent over $\C$, then $s \ge 1$ and $H_i \in \C[x_1,x_2,\ldots,x_s]$ 
for each $i$.
\end{theorem}

\begin{proof}
Just as in the proof of theorem \ref{qthrk2}, we may assume that $H_1 = H_2 = \cdots = H_s = 0$ 
and $H_{s+1}, H_{s+2}, \ldots, H_n$ are linearly independent over $\C$ for some $s \ge 1$.
If $1$ is linearly dependent over $\C$ of $H_{s+1}, H_{s+2}, \ldots, H_n$, then we may additionally
assume that $H_n = 1$. If $H_n = 1$, then define $\tilde{n} = n + 1$ and let $\tilde{H}$ be the 
homogeneization of $H$ as in proposition \ref{qthmg}. If $H_n = 1$, then define 
$\tilde{n} = n + 2$ and let $\tilde{H}$ be the homogeneization of $(H,1)$ as in proposition 
\ref{qthmg}. Let $\tilde{x} = (x_1, x_2, \ldots, x_{\tilde{n}})$. Since $(x,x_{n+1}) + (H,1)$
is a quasi-translation as well, it follows from proposition \ref{qthmg} that $\tilde{x} + \tilde{H}$
is also a quasi-translation, and that $\rk \jac_{\tilde{x}} \tilde{H} \le \rk \jac H + 1 = 2$.

One can easily verify that $\gcd\{\tilde{H}_1, \tilde{H}_2, \ldots, \tilde{H}_{\tilde{n}}\} = 1$
and that $\tilde{H}_{s+1}, \tilde{H}_{s+2}, \allowbreak \ldots, \tilde{H}_{\tilde{n}-1}$ are 
linearly independent over $\C$. Hence $\rk \jac \tilde{H} = 2$ on account of theorem \ref{qthrk1}.
By interchanging coordinates $s+1$ and $\tilde{n}$, we can deduce from theorem \ref{qthrk2}
that $s \ge 1$ and that $\tilde{H}_i \in \C[x_1,x_2,\ldots,\allowbreak x_s,x_{\tilde{n}}]$. 
Hence $H_i \in \C[x_1,x_2,\ldots,x_s]$ for each $i$.
\end{proof}

\begin{corollary} \label{qtdimsmall}
Assume $x + H$ is a quasi-translation in dimension $n$.
\begin{enumerate}[\upshape (i)]
 
\item If $n \le 3$, then the linear span of the image of $H$ has dimension 
at most $\max\{n-1,1\}$.

\item If $H$ is homogeneous and $n \le 4$, then the linear span 
of the image of $H$ has dimension at most $\max\{n-2,1\}$.

\end{enumerate}
In particular, the dimension of the linear span of the image of $H$ is at most
$2$ if either {\upshape (i)} or {\upshape (ii)} is satisfied.
\end{corollary}

\begin{proof}
Notice that (i) follows from (ii) by applying (ii) on the homogeneization of $H$
as in proposition \ref{qthmg}. So assume that $H$ is homogeneous and $n \le 4$.
Then it follows from theorem \ref{qthrk1} that either $\rk \jac H \le 1$ or 
$2 \le \rk \jac H \le n - 2 \le 2$. So $\rk \jac H \le 2$.

Suppose first that $\rk \jac H \le 1$. Then we can deduce from theorem \ref{qthrk1} that
the linear span of the image of $H$ is generated by a vector $c$, so that it has 
dimension at most $1$. Suppose next that $\rk \jac H = 2$. Then we can deduce from 
theorem \ref{qthrk2} that the linear span of the image of $H$ has dimension at most $n - 2$.
\end{proof}

Notice that example \ref{exampel} (in dimension 4) and its homogeneization (in dimension 5)
show that the bounds on $n$ in (i) and (ii) of the above corollary are sharp.
Example \ref{example} shows that for homogeneous quasi-translations, the linear span of the 
image of $H$ may have dimension $n$.

Corollary \ref{qtdimsmall} tells what the dimension of the linear span of $H$ is, but not
how $H$ looks. Since the dimension of the linear span of $H$ is at most $2$, we may
assume that $H_1 = H_2 = \cdots = H_{n-2} = 0$ on account of proposition \ref{qtconj}, 
so that we can apply the following theorem.

\begin{theorem} \label{qtform}
Assume $x + H$ is a quasi-translation, such that $H_1 = H_2 = \cdots = H_{n-2} = 0$.
Then $H$ is of the form 
$$
H = \big(0^1,0^2,\ldots,0^{n-2},b\,g,a\,g\big)
$$
where $g \in \C[x_1,x_2,\ldots,x_{n-2},a\,x_{n-1} - b\,x_n]$ 
and $a, b \in \C[x_1,x_2,\ldots,x_{n-2}]$.
\end{theorem}

\begin{proof}
Let $g = \gcd\{H_{n-1},H_n\}$ and take  $b = g^{-1} H_{n-1}$ and 
$a = g^{-1} H_{n-1}$. We first show that $a, b \in \C[x_1,x_2,\ldots,x_{n-2}]$. 
We distinguish two cases.
\begin{itemize}

\item \emph{$\rk \jac H \le 1$.} \\
If $H_{n-1}$ and $H_n$ are linearly independent over $\C$, then we can deduce from theorem 
\ref{qtrk1} that $H_{n-1}, H_n \in \C[x_1,x_2,\ldots,x_{n-2}]$, so that $a, b \in 
\C[x_1,x_2,\allowbreak\ldots,\allowbreak x_{n-2}]$ as well. 
If $H_{n-1}$ and $H_n$ are linearly dependent over $\C$, but not both equal to zero,
then $a, b \in \C$ because $\gcd\{a,b\} = 1$ and $a H_{n-1} - b H_n = 0$ is a linear dependence of  
$H_{n-1}$ and $H_n$ over $\C[x_1,x_2,\ldots,x_{n-2}]$. 

\item \emph{$\rk \jac H = 2$.} \\
Let $g = \gcd\{H_{n-1},H_n\}$. From proposition \ref{irred}, it follows that $x + g^{-1}H$
is a quasi-translation as well. Furthermore, $g$ is an invariant of $x + g^{-1}H$.
Let $\tilde{H}$ be the homogeneization of $g^{-1}H$. Just like in the last paragraph 
of the proof of theorem \ref{qtrk1}, we can deduce that $b = g^{-1} H_{n-1}$ and 
$a = g^{-1} H_{n-1}$ are contained in $\C[x_1,x_2,\ldots,x_{n-2}]$. 

\end{itemize}
Since $g$ is an invariant of $x + g^{-1}H$, it follows from \eqref{fxtHeqv} in proposition
\ref{qtprop} that 
\begin{equation} \label{geq}
b \parder{}{x_{n-1}} g + a \parder{}{x_{n}} g = 0
\end{equation}
If we express $g$ as a polynomial in 
$\C(x_1,x_2,\ldots,x_{n-2})[a\,x_{n-1} - b\,x_n,x_i]$ for some $i \in \{n-1,n\}$, 
then we can deduce from equation \eqref{geq} that 
$g \in \C(x_1,x_2,\ldots,\allowbreak x_{n-2})[a\,x_{n-1} - b\,x_n]$. 

Now let $g^{(k)}$ be the part of degree $k$ with respect to $(x_{n-1},x_n)$ of 
$g$. Then $g^{(k)}$ is a polynomial which is the product of an element $c_k$ of 
$\C(x_1,x_2,\ldots,x_{n-2})$ and $(a\,x_{n-1} - b\,x_n)^k$. 
Since $\gcd\{a,b\} = 1$, we can deduce from Gauss' lemma that 
$c_k \in \C[x_1,x_2,\ldots,x_{n-2}]$ for all $k$. 
So $g \in \C[x_1,x_2,\ldots,x_{n-2},a\,x_{n-1} - b\,x_n]$. 
\end{proof}

The following results follow immediately from corollary \ref{qtdimsmall} and theorems
\ref{qtrk1} and \ref{qtform}.

\begin{corollary}[Z. Wang] \label{qt3}
Assume $x + H$ is a quasi-translation in dimension $n \le 3$ over $\C$. Then there 
exists an invertible matrix $T$ over $\C$ such that $T^{-1}H(Tx)$ is of the form
\begin{enumerate}[\upshape (1)]

\item $(0,g)$, where $g \in \C[x_1]$, if $n = 2$,

\item $(0,b\,g,a\,g)$, where $g \in \C[x_1,a\,x_2 - b\,x_3]$ and $a,b \in \C[x_1]$,
      if $n = 3$.

\end{enumerate}
\end{corollary}

\begin{corollary} \label{qth4}
Assume $x + H$ is a homogeneous quasi-translation in dimension $n \le 4$ over $\C$. Then there 
exists an invertible matrix $T$ over $\C$ such that $T^{-1}H(Tx)$ is of the form
\begin{enumerate}[\upshape (1)]

\item $(0,0,g)$, where $g \in \C[x_1,x_2]$, if $n = 3$,

\item $(0,0,b\,g,a\,g)$, where $g \in \C[x_1,x_2,a\,x_3 - b\,x_4]$ and $a,b \in \C[x_1,x_2]$,
      if $n = 4$.
 
\end{enumerate}
\end{corollary}

Theorem \ref{qthrk2} was proved first by Gordan and N{\"o}ther in \cite{gornoet}.
Corollary \ref{qt3} was proved first by Wang in \cite{wang}.

\begin{remark}
If $x + H$ is a homogeneous quasi-translation such that $\rk \jac H = 2$, then
$x + gH$ does not need to be a quasi-translation. For homogeneous maps $H$ with 
$\rk \jac H = 2$, $x + g^{-1}H$ is a quasi-translation for some polynomial $g$, if and 
only if $\jac H \cdot H = \tr \jac H \cdot H$, i.e.\@ at most one of the eigenvalues of 
$\jac H$ is nonzero, see \cite[Th.\@ 4.5.2]{homokema}. If $\jac H \cdot H = 0$, then
it follows from \eqref{qtnilp} in proposition \ref{qtprop} that $\jac H$ is nilpotent, 
i.e.\@ all eigenvalues of $\jac H$ are zero.
\end{remark}

\section{Results and questions}

Gordan and N{\"o}ther classified all homogeneous polynomials with singular Hessians
in dimension five as follows.

\begin{theorem}[Gordan and N{\"o}ther] \label{gndim5}
Assume $h \in \C[x]$ is a homogeneous polynomial in dimension $n = 5$. 
If $\det \hess h = 0$ and $h$ does not satisfy Hesse's theorem, then there 
exists an invertible matrix $T$ over $\C$ such that $h(Tx)$ is of the form
$$
h(Tx) = f\big(x_1,x_2,a_1 x_3 + a_2 x_4 + a_3 x_5\big)
$$
where $f \in \C[x_1,x_2,x_3]$ and $a_1,a_2,a_3 \in \C[x_1,x_2]$.
\end{theorem}

This result can also be found in \cite{franch}, \cite[\S 4]{garrep}, and as 
\cite[Th.\@ 4.1]{hmgqt5dim}. Theorem \ref{gndim5} has a non-homogeneous variant, 
namely \cite[Th.\@ 3.3]{singhess}, which is as follows.

\begin{theorem} \label{gndim3}
Assume $h \in \C[x]$ is a polynomial in dimension $n = 3$. 
If $\det \hess h = 0$ and $h$ does not satisfy Hesse's theorem, then there 
exists an invertible matrix $T$ over $\C$ such that $h(Tx)$ is of the form
$$
h(Tx) = a_1 + a_2 x_2 + a_3 x_3
$$
where $a_1,a_2,a_3 \in \C[x_1]$.
\end{theorem}

The results of theorems \ref{hesspos} and \ref{gndim3} for polynomials
in $n$ variables with singular Hessians can be generalized to polynomials with
Hessian rank $n - 1$, as has been done in Theorems 5.3.5 and 5.3.10 in \cite{homokema}.
But the results of theorem \ref{hesspos} and \cite[Th.\@ 3.5]{singhess} admit a similar
genralization. Furthermore, one may replace $\C$ by any field of characteristic zero
in all of the above. These results will appear in a future paper by the author.

In dimension four, only those $h$ for which there exists a nonzero 
$R \in \C[y_1,y_2,\allowbreak y_3,y_4]$ with $R(\grad h) = 0$, such that
$$
c_1 \parder{R}{y_1} + c_2 \parder{R}{y_2} + c_3 \parder{R}{y_3} + c_4 \parder{R}{y_4} = 0
$$
for some nonzero $c \in \C^4$, are known. See \cite[Th.\@ 5.3.3]{homokema} or 
\cite[Th.\@ 3.5]{singhess}. This leads to the following question.

\begin{problem} \label{m}
Assume $h \in \C[x_1,x_2,x_3,x_4]$ such that $\det \hess h = 0$. Does there exist
a nonzero $R \in \C[y_1,y_2,y_3,y_4]$ such that $R(\grad h) = 0$ and
$$
c_1 \parder{R}{y_1} + c_2 \parder{R}{y_2} + c_3 \parder{R}{y_3} + c_4 \parder{R}{y_4} = 0
$$
for some $c \in \C^4$?
\end{problem}

The one who first solves this problem receives a bottle of Joustra Beerenburg (Frisian spirit).
Problem \ref{mm} below generalizes problem \ref{m}. Indeed, suppose that $f$ is a counterexample to
problem \ref{m}. Then it is known that the $R$ in problem \ref{m} cannot be homogeneous, which
is in fact the equivalent of problem \ref{mm} in dimension four instead of five. From that, one 
can deduce that $f + x_5$ is a counterexample to problem \ref{mm}.

\begin{problem} \label{mm}
Assume $h \in \C[x_1,x_2,x_3,x_4,x_5]$ such that $R(\grad h) = 0$ for some
nonzero homogeneous $R \in \C[y_1,y_2,y_3,y_4,y_5]$. Is
$$
c_1 \parder{R}{y_1} + c_2 \parder{R}{y_2} + c_3 \parder{R}{y_3} + c_4 \parder{R}{y_4} + 
c_5 \parder{R}{y_5} = 0
$$
for some $c \in \C^5$?
\end{problem}
 
Quasi-translations in dimension four and homogeneous quasi-translations in 
dimension five are not classified. For quasi-translations $x + H$ 
in dimension five, one can show that $\rk \jac H \le 2$ if $H$ is constructed as
$H = (\grad R)(\grad h)$ for a {\em homogeneous} $h \in \C[x_1,x_2,x_3,x_4,x_5]$ 
and {\em any} $R \in \C[y_1,y_2,y_3,y_4,y_5]$ such that $R(\grad h) = 0$. This
is what Gordan and N{\"o}ther actually did to get theorem \ref{gndim5}, see also
the proof of \cite[Th.\@ 4.1]{hmgqt5dim}.

The quasi-translation $x + \tilde{H}$ in example \ref{exampel} has the 
property that the linear span of the image of $\tilde{H}$ has
dimension $n$. Hence one can ask the following.

\begin{problemm}
Assume $x + H$ is a quasi-translation in dimension $n=4$
such that $H = (\grad R)(\grad h)$ for some $h \in \C[x]$ and
an $R \in \C[y]$ satisfying $R(\grad h) = 0$. Is the dimension of 
the linear span of the image of $H$ less than $n=4$?
\end{problemm}

If you think that you have seen problem \ref{m} before, you are quite
right. Both problems \ref{m} are equivalent.

The quasi-translation of example \ref{exampel} in dimension $n=4$ is not 
homogeneous, but the quasi-translation of example \ref{example} in dimension $n=6$
is indeed homogeneous, and the linear span of the image of $H$ has dimension $n$ as well.
Since the linear span of the image of $H$
has dimension less than $n$ if $n \le 3$ and dimension less than $n-1$
if $H$ is homogeneous and $3 \le n \le 4$, we have the following question.

\begin{problem} \label{hqt5}
Assume $x + H$ is a homogeneous quasi-translation in dimension $n=5$.
Is the dimension of the linear span of the image of $H$ less than $5$?
\end{problem}

Also for this problem, the one who first solves it receives a bottle 
of Joustra Beerenburg (Frisian spirit).

For a counterexample to problem \ref{hqt5}, we would have $\rk \jac H = 3$.
Gordan and N{\"o}ther divide the homogeneous quasi-translations $x + H$
in dimension five with $\rk \jac H = 3$ in two groups, which they indicate
as `Fall a)' and `Fall b)'. They prove that for all `Fall a)' 
quasi-translations $x + H$, the linear span of the image of $H$ is indeed 
less than five. The homogeneization of example \ref{exampel} appears to
be a `Fall b)' quasi-translation. See \cite{hmgqt5dim}, in particular 
section 5 of it, for more information about `Fall b)' quasi-translations 
$x + H$ and problem \ref{hqt5}.

As observed earlier, $\jac H$ is nilpotent if $x + H$ is a quasi-translation,
because of \eqref{qtnilp} in proposition \ref{qtprop}.
Taking about nilpotent matrices, one can wonder how polynomials $h$ with
a nilpotent Hessian look if the dimension or the Jacobian rank is small.

\begin{theorem}
Assume $h \in \C[x]$ is not necessarily homogeneous, such that $\hess h$ is nilpotent. 
Then Hesse's theorem holds in the following cases.
\begin{enumerate}[\upshape (i)]

\item $\rk \hess h \le 2$ or $n \le 4$,

\item $h$ is homogeneous and $n = 5$ (in addition to $\rk \hess h \le 3$, in particular
$n \le 4$, for which $\det \hess h = 0$ is sufficient).

\end{enumerate}
but there is a counterexample in dimension $6$ and a homogeneous counterexample in 
dimension $10$.
\end{theorem}

\begin{proof}
(i) and (ii) follow from \cite[Th.\@ 5.7.1]{homokema}. See the beginning 
of \cite[\S 5.7]{homokema} for how the counterexamples can be obtained.
\end{proof}


\begin{thebibliography}{9}

\bibitem{homokema} Michiel de Bondt, Homogeneous Keller maps. 
Ph.D. thesis, Radboud University, Nijmegen, 2009. \\
\verb+http://webdoc.ubn.ru.nl/mono/b/bondt_m_de/homokema.pdf+

\bibitem{hmgqt5dim} Michiel de Bondt, Homogeneous quasi-translations in dimension $5$,
arXiv:1501.04845, 2015.

\bibitem{singhess}
M.C. de Bondt and A.R.P. van de Essen, Singular Hessians, J. Algebra 282 (2004), 
no.\@ 1, 195--204. 

\bibitem{debunk}
M.C. de Bondt,
Quasi-translations and counterexamples to the homogeneous dependence problem. 
Proc.\@ Amer.\@ Math.\@ Soc.\@ 134 (2006), no.\@ 10, 2849--2856 (electronic).

\bibitem{arnobook}
A. van den Essen, Polynomial Automorphisms and the Jacobian Conjecture.
Birkh{\"a}user, Berlin.

\bibitem{franch}
Alfredo Franchetta, Sulle forme algebriche di $S_4$ aventi l'hessiana indeterminata. 
Rend.\@ Mat.\@ e Appl.\@ (5) 14 (1954), 252--257. 

\bibitem{garrep}
Alice Garbagnati and Flavia Repetto, A geometrical approach to Gordan-N{\"o}ther's 
and Franchetta's contributions to a question posed by Hesse. 
Collect.\@ Math.\@ 60 (2009), no.\@ 1, 27--41. 

\bibitem{gornoet}
P. Gordan and M. N{\"o}ther, {\"U}ber die algebraische Formen,
deren Hes\-se'sche Determinante identisch verschwindet.
Math.\@ Ann.\@ 10 (1876), 547--568.

\bibitem{gnlossen}
C. Lossen, When does the Hessian determinant vanish identically? 
(On Gordan and N{\"o}ther's proof of Hesse's claim). 
Bull.\@ Braz.\@ Math.\@ Soc.\@ (N.S.) 35 (2004), no.\@ 1, 71--82. 

\bibitem{schinzel}
A. Schinzel, Selected topics on polynomials. The Univ.\@ of Michigan Press, 
Ann Arbor, MI, 1982.

\bibitem{wang}
Zhiqing Wang, Homogeneization of locally nilpotent derivations and an application 
to $k[X,Y,Z]$. J. Pure Appl.\@ Algebra, 196 (2005), no.\@ 2--3, 323--337.

\end{thebibliography}
\end{document}